\documentclass[11 pt]{article}
\usepackage{verbatim}
\usepackage{amsmath}
\usepackage{amsfonts}
\usepackage{amsthm}
\usepackage{enumerate}
\usepackage{graphicx}
\usepackage{float}
\usepackage{ragged2e}
\usepackage{mathtools}
\usepackage{url}
\usepackage[top=2.5cm, bottom=2.5cm, left=2.5cm, right=2.5cm]{geometry}
\usepackage{caption}
\usepackage{subcaption}
\usepackage{enumitem}
\usepackage{pgfplots}

\DeclarePairedDelimiter\floor{\lfloor}{\rfloor}
\DeclarePairedDelimiter\ceiling{\lceil}{\rceil}

\newtheorem{theorem}{Theorem}[section]
\newtheorem*{theorem*}{Theorem}
\newtheorem{Lemma}[theorem]{Lemma}
\newtheorem*{Lemma*}{Lemma}
\newtheorem{proposition}[theorem]{Proposition}
\newtheorem*{proposition*}{Proposition}

\newtheorem{maintheorem}{Theorem}

\theoremstyle{definition}
\newtheorem{definition}[theorem]{Definition}

\title{On random presentations with fixed relator length}
\author{C. J. Ashcroft, Colva M. Roney-Dougal}

\begin{document}
\maketitle
\begin{abstract}
	The \emph{standard $(n, k, d)$ model of random groups} is a model where
        the 
relators are chosen randomly from the set of cyclically reduced words 
of length $k$ on an $n$-element generating set. Gromov's density model
of random groups considers the case where $n$ is fixed, and $k$ tends
to infinity. We instead fix $k$, and let $n$ tend to
infinity. We prove that for all $k \geq 2$
at 
density $d > 1/2$ a random group in this model is trivial
or cyclic of order two, whilst for  $d <  \frac{1}{2}$ such  a random group
is infinite and hyperbolic. In addition we
show that for $d<\frac{1}{k}$ such a random group  is free, and that this
threshold is sharp. These extend known results for the triangular ($k
= 3$) and square ($k = 4)$ models of random groups.
\end{abstract}

\section{Introduction and statement of results} 
Gromov's density model of random groups is a famous construction in
modern group theory, introduced in \cite{gromov1993geometric} to
answer the question of what a ``generic" group looks like. Models of
random groups have also been used to construct exotic groups, as in
\cite{gromov2003random}.  

Let $n\geq 2$, $k\geq 2$, and $0<d<1$. Randomly uniformly select a set
$R$ of distinct cyclically reduced words of length $k$ over the
alphabet $\{a_{1},\hdots, a_{n}\}$ of size $N=\lfloor (2n-1)^{kd}
\rfloor$ from amongst all such sets. Then let $G=\langle a_{1},\ldots,
a_{n}\;\vert\; R\rangle$. The group $G$ is a \emph{random group} in
the \emph{standard (n,k,d) model} of random groups. If we take the set
$R$ to contain only positive words over the alphabet (that is, words
where no letters are inverses of the generators), we obtain the
\textit{positive $(n,k,d)$ model}, introduced by
Odrzyg{\'o}{\'z}d{\'z}. By  \emph{an $(n, k, d)$ model} we mean either
the standard or the positive $(n, k, d)$ model. 

Gromov showed (see also \cite{Ollivier2004}) that for any fixed $n
\geq 2$, with probability
tending to $1$ as $k$ tends to $\infty$, a random group in the
standard $(n,k,d)$ model is trivial or isomorphic to $\mathbb{Z}_{2}$
for $d>\frac{1}{2}$, and infinite, hyperbolic, and torsion free for
$d<\frac{1}{2}$, so that $d=\frac{1}{2}$ is a sharp phase transition
for the density model as $k \rightarrow \infty$.

\begin{definition}
Let $k\in\mathbb{N}_{\geq 2}$, and let $0< d< 1$. Let $\mathcal{P}$ be
a property of groups preserved by isomorphism, and let $\mathcal{M}(n,
k, d)$ be an $(n, k, d)$ model. 
 If \begin{equation*}
	\lim_{n\rightarrow \infty}P(\mbox{A random group in }\mathcal{M}(n,k,d)\mbox{ satisfies }\mathcal{P})=1,
\end{equation*}
we will say that \emph{at density $d$, asymptotically almost surely
  (a.a.s.), a random group in the $\mathcal{M}(n, k, d)$ model
  satisfies $\mathcal{P}$.} We shall also call the standard $(n, k,
d)$-model as $n \rightarrow \infty$ the \emph{standard $k$-angular
  model at density $d$}, with a comparable definition of the \emph{positive
  $k$-angular model}. 
\end{definition}

{\.Z}uk studied the standard $3$-angular model, 
also known as the \emph{triangular} model, showing in
\cite{zuk2003property} that at density $d<\frac{1}{2}$ a.a.s. a group
in the triangular model is hyperbolic. He also showed that at density $d<\frac{1}{3}$
a.a.s. a group in the triangular model is isomorphic to $F_{m}$ for
some $m\in\mathbb{N}$, and that at density $d > \frac{1}{3}$ such a
group satisfies Property (T) (and so is not free). 

The standard $4$-angular model is also known as the the \emph{square}
model. This was studied more recently by Odrzyg{\'o}{\'z}d{\'z} in 
\cite{odrzygozdz2014square}, who showed at at $d>\frac{1}{2}$ a.a.s. a
group in the square model is cyclic of order two,
whilst at $d<\frac{1}{2}$ a.a.s. such a group is
hyperbolic. Furthermore, at density $d<\frac{1}{4}$ a.a.s. a group in
the square model  is isomorphic to $F_{m}$ for some
$m\in\mathbb{N}$.

In
\cite{odrzygozdz2014square}, Odrzyg{\'o}{\'z}d{\'z}  also introduced
and studied the positive square model, showing that at density
$d>\frac{1}{2}$ a.a.s. a group in this model is cyclic of order four, 
whilst at $d<\frac{1}{2}$ a.a.s. such a group is
hyperbolic. Furthermore, at $d<\frac{1}{4}$ a.a.s. a random group in
the positive square model  is isomorphic to $F_{m}$ for some $m\in\mathbb{N}$.

This paper generalises this work of {\.Z}uk  and
Odrzyg{\'o}{\'z}d{\'z} to all values of $k$. 
For the standard $k$-angular model we have the following:
\begin{maintheorem}\label{mainthm2} Let $k\geq 2$, and let $G$ be a
  random group in the standard $(n, k, d)$ model.
	\begin{enumerate}[label=\roman*)]
			\item If $d> \frac{1}{2}$  then asymptotically almost surely $G$ is trivial ($k$ odd) or cyclic of order two ($k$ even).
			\item If $d<\frac{1}{2}$  then asymptotically
                          almost surely $G$ is hyperbolic, infinite
                          and torsion-free.
	\end{enumerate}
\end{maintheorem}

As a side-effect of our proof techniques, we also show the following. 

\begin{maintheorem}\label{mainthm1} Let $k\geq 2$, and let $G$ be a
  random group in the positive $(n, k, d)$ model. 
	\begin{enumerate}[label=\roman*)]
			\item If $d> \frac{1}{2}$  then asymptotically almost surely $G$ is
                          cyclic of order $k$.
			\item If  $d<\frac{1}{2}$ then asymptotically
                          almost surely $G$ is hyperbolic, infinite
                          and torsion-free.
	\end{enumerate}
\end{maintheorem}

Furthermore, we consider freeness. 
\begin{maintheorem}\label{mainthm3} Let $k\geq 2$, and let $G$ be a
  random group in an $(n, k, d)$ model. If
  $d<\frac{1}{k}$ then asymptotically almost surely $G$ is isomorphic to $F_{m}$ for some
  $m\in\mathbb{N}$, whilst if $d > \frac{1}{k}$ the group $G$ is asymptotically almost surely not isomorphic to a nontrivial free group.
			\end{maintheorem}

The paper is structured as follows. In Section \ref{part i) proof} we
prove Theorems~\ref{mainthm2}(i)  and \ref{mainthm1}(i).  In
Section \ref{part ii) proof} we prove part (ii) of the same theorems. 
In Section \ref{sec: freeness} we prove Theorem \ref{mainthm3}, and
conclude with a short appendix giving a new proof of a known result on
random bipartite graphs. 


\section{Groups at density greater than $1/2$}\label{part i) proof}
In this section we will use random graphs to understand the
relationship between the generators of a random group, and in particular
to show that asymptotically almost surely in a random group in a 
$k$-angular model at density greater than ${1}/{2}$, all generators are equal.

\begin{definition}
	The \emph{random bipartite graph} $\Gamma(a,b,E(a,b))$ is a
        graph obtained by sampling uniformly at random from the set of
        all bipartite graphs $\Gamma$ with parts $V_1$ and $V_2$ such that $\vert V_{1}\vert=a$, $\vert V_{2} \vert=b$, and $\vert E(\Gamma)\vert =E(a,b)$. 

\end{definition}

The following result is
similar to a famed theorem of Erd\"{o}s and R\'enyi.

\begin{Lemma}[{\cite[Theorem
1]{palasti1963connectedness}}]\label{nnconn}
	Let $a>0$, let $\varepsilon \in (0, 1)$, and let $P(a,a,a^{1+\varepsilon})$ denote the probability that a random bipartite graph $\Gamma(a,a,a^{1+\varepsilon})$ is connected. Then 
	\begin{equation*}
		\lim\limits_{a\rightarrow\infty}P(a,a,a^{1+\varepsilon})=1.
	\end{equation*}
\end{Lemma}

The next  result follows with only a little work from 
 \cite[Theorem 9]{johansson2012giant}, where the corresponding result
 is shown for an Erd\"os-R\'enyi random bipartite graph, 
but for convenience we 
include a proof in the Appendix to this paper.

\begin{Lemma}\label{red_connected}
Let $d \in (1/2, 1)$, and let 
$\Gamma$ be a  $\Gamma(n^m, n^{m+1}, n^{(2m+1)d})$ random bipartite
graph. 
Then asymptotically almost surely all vertices in $V_{1}$ 
are in the same connected component of $\Gamma$.
\end{Lemma}

For the remainder of this paper, 
let $C_{k,n}$ be the set of cyclically reduced words of length
$k$ in $F_{n}$, and $C_{k,n}'$ be the set of positive words of length
$k$ in $F_{n}$. We now prove a theorem that immediately implies
Theorem~\ref{mainthm1}(i). 

\begin{theorem}\label{posmodtriv}
Let $k\geq 2$ and $d>\frac{1}{2}$. Let $G=\langle a_1,  \ldots, a_n
\;\vert\;R\rangle$ be a random group in the positive $(n,k,d)$
model. Asymptotically almost surely
 the group $G$ is isomorphic to $\mathbb{Z}_{k}$, and 
$a_1 =_G \cdots =_G a_n$.
\end{theorem}

\begin{proof}
Let
\begin{equation*}	
\begin{split}
H_{n}^{l}=&\{a_{i_{1}}\hdots a_{i_{\floor*{\frac{k}{2}}}}\;\vert\; a_{i_{j}}\in \{a_{1},\hdots, a_{n}\}\},\\
H_{n}^{u}=&\{a_{i_{1}}\hdots a_{i_{\ceiling*{\frac{k}{2}}}}\;\vert\; a_{i_{j}}\in \{a_{1},\hdots, a_{n}\}\},
\end{split}
\end{equation*}
and notice that  $C_{k,n}'= \{xy\;\vert \;x\in H_{n}^{l},y\in H_{n}^{u}\}$. 

Define a random bipartite graph $\Gamma$ with parts $V_1$ and $V_2$ as follows. Take as $V_1$ the elements of $H_{n}^{l}$ and as $V_2$ the elements of $H_{n}^{u}$. For $x\in H_{n}^{l}$ and $y\in H_{n}^{u}$, draw an edge between $x$ and $y$ if $xy\in R$.
 Then $\Gamma$  is a $\Gamma(n^{\floor*{\frac{k}{2}}}, n^{\ceiling*{\frac{k}{2}}}, n^{kd})$ random bipartite graph. 

Let $\varepsilon=2d-1>0$. 
If $k$ is even then $n^{kd}=(n^{\frac{k}{2}})^{1+\varepsilon}$, so by Lemma
 \ref{nnconn} a.a.s. this graph is connected.
If $k$ is odd then by Lemma \ref{red_connected} a.a.s. 
there exists a connected component of $\Gamma$  spanning $V_1$.

Hence in both cases a.a.s. there is a path of even length 
between any two vertices in $V_1$. An edge between $x$ and $y$ 
corresponds to $xy\in R$, and so $x=_{G}y^{-1}$. Hence a path of even
length corresponds to equality in $G$, and therefore a.a.s. all
elements in $H_{n}^{l}$ are equal in $G$. In particular for distinct
$a_{i},a_{j}$, 
a.a.s. $a_{i}^{\floor*{\frac{k}{2}}}=_{G}a_{i}^{\floor*{\frac{k}{2}}-1}a_{j}$, and so $a_{i}=_{G}a_{j}$.

All relators are positive words of length $k$, so a.a.s. $G$ is isomorphic to $\mathbb{Z}_{k}$.
\end{proof}

We now wish to show that a.a.s. a random group in the standard $(n, k, d)$
model at density $d>\frac{1}{2}$ contains enough positive words as
relators to force the generators to be equal. We first record a
well-known probability estimate. 
\begin{Lemma}[{\cite[Corollary 1.1]{serfling1974}}]\label{serfling}
	Sample without replacement from a finite list $x_{1},\hdots,x_{M}$. Let $X_{1},\hdots,X_{m}$ be these samples. Define $S_{m}=\sum_{i=1}^{m}X_{i}$, $\mu=\frac{1}{M}\sum_{i=1}^{M}x_{i}$. Also let $y_{0}=\min x_{i}$, $y_{1}=\max x_{i}$. Then
	\begin{equation*}
		P[S_{m}\geq m\mu+mt]\leq exp[-2mt^{2}/(y_{1}-y_{0})^{2}].
	\end{equation*}
\end{Lemma}
We now prove the following, extending from $k=4$ in \cite[Lemma 2.9]{odrzygozdz2014square}.

\begin{Lemma}\label{serflingimplies}
	Let $k, n\geq 2$,  $d>\dfrac{1}{2}$, and let $G=\langle
        X\;\vert\; R\rangle$ be a random group in the standard $(n,k,d)$ model. With probability tending to $1$ as $n$ tends to $\infty$, $\vert R\cap C_{k,n}' \vert \geq n^{kd'}$ for any $\frac{1}{2}<d'< d$.
\end{Lemma}

\begin{proof}
Randomly selecting $R$ is equivalent to sampling relators from $C_{k, n}$ without replacement. Let $\frac{1}{2}<d'< d$. We show that
\begin{equation*}
P[\vert R\cap C_{k,n}' \vert < \frac{1}{2^{k}+1}(2n-1)^{kd}]=P[\vert R\cap( C_{k,n}\setminus C_{k,n}') \vert\geq \frac{2^{k}}{2^{k}+1}(2n-1)^{kd}]\rightarrow 0\mbox{ as }n\rightarrow\infty.
\end{equation*}
As $n \rightarrow \infty$, with $k$ and $d$ fixed,  $$\frac{1}{2^{k}+1}(2n-1)^{kd}\geq\frac{(2n-1)^{kd-kd'}}{2^{k}+1}n^{kd'}\geq n^{kd'}.$$
 and the result will follow. 
 
  Consider the following. For relators $r_{1},\hdots, r_{(2n-1)^{kd}}$ define the random variable $X_{i}$ by
\begin{equation*}
	X_{i}=\begin{cases}
		1\mbox{ if }r_{i}\in C_{k,n}\setminus C_{k,n}' ,\\
		0\mbox{ otherwise.}
	\end{cases}
\end{equation*}
This equivalent to sampling without replacement from $\vert C_{k,n}\setminus C_{k,n}'\vert $ ones and $\vert C_{k,n}'\vert $ zeros. Notice that $|C_{k, n}| < 2n(2n-1)^{k-1}$, whilst $|C'_{k, n}| = n^k$.   In the notation of Lemma \ref{serfling}, 
\begin{equation*}
\mu=\frac{\vert C_{k,n}\setminus C_{k,n}'\vert}{\vert C_{k,n}\vert} = 1 - \frac{|C'_{k, n}|}{|C_{k,n}|}< 1 - \frac{n^{k}}{(2n)^{k}} < 1 - \frac{1}{2^{k}},
\end{equation*}
$y_{0}=0,y_{1}=1$, and $m=(2n-1)^{kd}$.
 Hence, letting $t=\frac{2^{k}}{2^{k}+1}-\mu>0$,
\begin{equation*}
 P[\vert R\cap C_{k,n}\setminus C_{k,n}' \vert\geq  \frac{2^{k}}{2^{k}+1}(2n-1)^{kd}]=P_{m}(t)\leq  exp[-2(2n-1)^{kd}t^{2}]\rightarrow 0\mbox{ as }n\rightarrow\infty.
\end{equation*}
\end{proof}

\begin{proof}[Proof of Theorem~\ref{mainthm2}(i)]
	By Lemma \ref{serflingimplies}, a.a.s. $G$ has at least
        $n^{kd'}$ positive relators for any $\frac{1}{2}<d'<d$. Hence
        by Theorem~\ref{posmodtriv}, a.a.s. $a_{i}^{k}=_{G}1$ for all
        generators $a_{i}$ of $G$, and $a_i = a_j$ for all $i$ and
        $j$. 
Also, a.a.s. there is a cyclic conjugate of a word of the form
$a_{i_{1}}\hdots a_{i_{k-1}}a_{i_{k}}^{-1}$ in $R$ -- there are at
least $k(2n-2)n^{k-1}$ such cyclic conjugates, and a proof similar to
Lemma \ref{serflingimplies} follows. 
Hence a.a.s $a_1^{k-2}=_G 1$, and so
$G$ is isomorphic to a cyclic group of order $2$ if $k$ is even, and
$1$ if $k$ is odd. 
\end{proof}

\section{Groups at density less than $1/2$}\label{part ii) proof}

In this section we shall prove Theorems~\ref{mainthm2}(ii) and
\ref{mainthm1}(ii). 

We first introduce the diagrams with which we shall be
working; our definitions follow 
\cite{Ollivier2004}. For a set
$X$, we write
$X^{\pm}$ to denote the (disjoint) union $X \cup X^{-1}$.

\begin{definition} Let $G=\langle X\;\vert\;R\rangle$ be a group. A
  \emph{van Kampen diagram} for $G$ is a planar,
simply-connected, finite $2$-complex, $A$ such that 
\begin{enumerate}[label=\roman*)]
	\item each $1$-cell in $A$ is labelled by an element in $X^\pm$,
	\item if $e$ is an oriented $1$-cell with the opposite orientation denoted $e^{-1}$, then the label of $e$ is the inverse of the label of $e^{-1}$,
	\item each $2$-cell $D$ of $A$ is oriented, and has a marked
          start point on $\partial D$. Reading along $\partial D$ from
          the start point, in the direction given by the orientation,
          yields a relator $r \in R$. We say that $D$ \emph{bears}
          $r$.  
 \end{enumerate}If $w\in F_{X}$ is freely reduced, $A$ is a van Kampen
 diagram for $G$, and there exists a $0$-cell in $\partial A$ such
 that reading clockwise along $\partial A$ and concatenating the
 labels of the $1$-cells in $\partial A$, $w$ is the word obtained,
 then $A$ is a van Kampen diagram \emph{for $w$}, and $w$ is  a
 \emph{boundary word} for $A$. (If $A$ is a sphere, then $w$ is the
 empty word.) 
\end{definition}

We will write $\vert A\vert$ for the number of $2$-cells in $A$, and $\vert \partial A\vert$ for the number of $1$-cells in $\partial A$ (or the length of a boundary word of $A$). We will switch freely between the words vertex and $0$-cell, edge and $1$-cell, and face, region and $2$-cell.

\begin{definition}
 Let $G$ be a group and $A$ a van Kampen diagram for $G$. 
The diagram $A$ is
 \emph{unreduced} if there exist regions $D_{1}$ and $D_{2}$ in $A$
bearing the same relator with opposite orientations, and with shared
edge representing the same letter in the relator (with respect to the
marked start points). 
A van Kampen
 diagram is \emph{reduced} if it is not unreduced. 
\end{definition}

Note that this definition of reduced is slightly weaker than the
standard definition, due to the marked start points. 
It is standard that any unreduced van Kampen diagram can be
transformed into a reduced van Kampen diagram without altering the
boundary word (as an element of $F(X)$). 

\begin{definition}
Let $A$ be a van Kampen diagram, with each edge labelled by a
generator or its inverse, bearing $N$ distinct relators
$r_{i_{1}},\hdots,r_{i_{N}}$. We construct the \emph{abstract 
  van Kampen diagram} from $A$ as follows. Each face bearing relator
$r_{i_{j}}$ is labelled with the number $j$, we record the orientation
of each face, and the starting point of the relator in the face
boundary. An abstract van Kampen diagram is
\emph{reduced} if it does not contain two regions with the same
relator number that have opposite orientations, share an edge, with
the common edge having the same position in the boundary with respect
to the marked start position. 
 For an abstract van Kampen diagram, $A$, we write
 $\vert \partial A\vert$ for the number of edges in the boundary of
 $A$, and $\vert A\vert$ for the number of $2$-cells in $A$.  
\end{definition}

The following applies to both van Kampen diagrams and abstract van Kampen diagrams.
\begin{definition}\label{spurdef}
A \emph{spur} in a diagram $A$ is an edge such that either its start or end vertex is of degree $1$.
 The diagram $A$ is \emph{spurless} if it has no spurs. 
 A \emph{filament} in $A$ is a non-spur edge, $e$, such that for
 all regions, $D$, in $A$, the intersection $e\cap \partial D=\emptyset$.
 An edge is \emph{non-filamentous} if is not a filament.   A vertex is
 \emph{exterior} if it lies on $\partial A$, and \emph{interior} if it
 is not exterior. Similarly, an edge is \emph{exterior} if it lies on
 $\partial A$, and \emph{interior} if it is not exterior.
\end{definition}

Notice that if all relators are assumed to be cyclically reduced, then
every interior edge is automatically non-filamentous. We shall
therefore implicitly make this assumption from now on. 

\begin{definition}
An abstract reduced van Kampen diagram $A$ is \emph{fulfillable} with
respect to the presentation $G=\langle X\;\vert\;R\rangle$ if there
exist relators $r_{i_{1}},\hdots, r_{i_{N}} \in R$ such that each
relator $r_{i_{j}}$ attaches to faces bearing $j$ (respecting start
vertex and orientation) and the result is a valid reduced van Kampen
diagram for $G$, i.e. there exist letters in $X^{\pm}$ that can label
spurs and filaments which gives rise to a reduced van Kampen diagram
for $G$. In this case, the relators $r_{i_{j}}$ are said to
\emph{fulfil} $A$. Relators $r_{i_{1}},\hdots ,r_{i_{m}}$ ($m\leq N$)
\emph{partially fulfil} $A$ if the assignments of $r_{i_{j}}$ to
regions bearing $j$ do not produce a contradiction.
\end{definition}

\begin{definition}
Let $A$ be an abstract van Kampen diagram. For a face
$f$, bearing edge $e$, let $(i(e,f),j(e,f))$ be the tuple where
$i(e,f)$ is the relator number $f$ bears, and $j(e,f)$ is the position of $e$
in $\partial f$ with respect to the marked start point. We order the tuples lexicographically, so that
$(i(e,f),j(e,f))>(i(e',g),j(e',g))$ if $i(e,f)>i(e',g)$ or if
$i(e,f)=i(e',g)$ and $j(e,f)>j(e',g)$. 

If an edge $e$ is interior
and incident to faces $f,g$ such that
$(i(e,f),j(e,f))>(i(e,g),j(e,g))$, 
then $e$ \emph{belongs} to $f$.  We do \emph{not} assign ownership 
of exterior edges.
Note that if $f \neq g$ then
$(i(e,f),j(e,f)) \neq (i(e,g),j(e,g))$ since otherwise
 either $f$ bears the same relator number in the
opposite orientation to $g$ and with equivalent start points 
(and so $A$ is not reduced), or a letter is its own inverse in $F_{n}$. 
For a face $f$, let $\omega (f)$ 
be the number of edges that belong to $f$, and for $i=1,\hdots, N$ let 
	\begin{equation*}
\omega_{i}=\max\limits_{f\; bearing\; i}\omega (f).		
	\end{equation*}
\end{definition}

\begin{Lemma}\label{pi inequality Olliv}
Let $A$ be a reduced abstract van Kampen diagram with
relator numbers $\{1, \ldots, N\}$ and with all faces of boundary length $k$. 
	For $i=1,\hdots,N$, let $p_{i}$ be the probability that $i$
        randomly chosen cyclically reduced words of length $k$
        partially fulfil $A$ (as relator numbers $1, \ldots, i$), 
        and let $p_{0}=1$. Then for $i=1,\hdots,N$,
\begin{equation*}
	\frac{p_{i}}{p_{i-1}}\leq (2n-1)^{-\omega_{i}}\frac{2n-1}{2n-2}.
\end{equation*}
\end{Lemma}
\begin{proof}
	Let $1\leq i\leq N$. Let us assume this is true for
        $p_{1},\hdots, p_{i-1}$, and assume that we have chosen
        relators $w_{1},\hdots, w_{i-1}$ that partially fulfil
        $A$. Now suppose that we have chosen $0\leq l<k$ letters of
        the relator $w_{i}$, and let $f$ be a face of $A$ bearing
        relator number $i$. We now choose the letter 
        corresponding to an edge $e$ of $f$. 

If $e$ belongs to $f$, then there exists another face, $g$, with
$(i(e,g),j(e,g))<(i(e,f),j(e,f))$. This means that either the edge
appears in some word $w_{j}$ for $j<i$, and so the label of $e$ is
already fixed, or it appears as an earlier edge in a face bearing the
same relator as $f$ (and so as an earlier edge in $f$), and therefore
the label of $e$ is also fixed. We choose letters to obtain a
cyclically reduced word: there are $2n$ choices for the first, $2n-1$
for the middle $k-2$ letters, and $2n-1$ or $2n-2$ choices for the
last letter.  Hence the probability of a randomly chosen letter being
valid is less than or equal to $\frac{1}{2n}\leq\frac{1}{2n-1}$ if
this is the first letter, $\frac{1}{2n-1}$ for the middle $k-2$
letters, and $\frac{1}{2n-1}$ or$\frac{1}{2n-2}$ for the final
letter. 

If $e$ does not belong to  $f$, then the probability that a
randomly chosen label for $e$ partially fulfils $A$ is at most $1$.

So having chosen $i-1$ random words partially fulfilling $A$, the probability
that the next 
 random word partially fulfils $A$ is less than or equal to 
$(2n-1)^{-(\omega (f)-1)}(2n-2)^{-1}$ for all $f$ bearing $i$.
Hence the probability, $p_{i}$, of $i$ randomly chosen words 
partially fulfilling $A$ satisfies
\begin{equation*}
p_{i}\leq p_{i-1} \min\limits_{f\;bearing\;i}\{(2n-1)^{-\omega (f)}\} \frac{2n-1}{2n-2}=p_{i-1}(2n-1)^{-\omega_{i}}\frac{2n-1}{2n-2},
\end{equation*}
and the result follows. 
\end{proof}

We  proceed by evaluating the probability that an abstract reduced van Kampen diagram satisfies the linear isoperimetric inequality or the probability that it can be fulfilled if it does not.
The following lemma is a slightly more precise version of 
\cite[Proposition 58]{ollivier2005}.

\begin{Lemma}\label{linisolliv}
	Let $n,k\geq 2$, $d<\frac{1}{2},$ and $\varepsilon >0$. Let
        $G=\langle X\;\vert\;R\rangle$ be a random group in the
        standard $(n,k,d)$ model. Any spurless abstract reduced van Kampen diagram, $A$, for $G$ either satisfies 
	\begin{equation*}
\vert \partial A\vert \geq k\vert A\vert\left(1-2d-2\varepsilon-
\frac{2}{k}(1 - \log_{(2n-1)}(2n-2))\right)
\end{equation*}
	or the probability it is fulfillable in $G$ is at most $(2n-1)^{-\varepsilon k}$.
\end{Lemma}
\begin{proof}
Let $A$ be a spurless abstract reduced van Kampen diagram, 
let $N$ be the number of
 distinct relators in $A$, so that $N\leq \vert A\vert$, and let
 $m_{i}$ be the number of faces bearing relator number $i$. We can assume
 without loss of generality that the $m_{i}$ are non-increasing. 

	Now, $k\vert A\vert$ counts each interior edge twice, and each non filamentous exterior edge once, but does not count any filaments. Also, 
\begin{equation*}
	\sum\limits_{f\; face\; of\; A}\hspace{-10 pt}\omega(f)	
\end{equation*}
counts each interior edge once. So 
	\begin{equation*}
			\vert \partial A\vert \geq \; k\vert A\vert - 2\sum_{f\;face\;of\;A}.\hspace{-10 pt}\omega (f).
	\end{equation*}
Let $p_i$ be as in Lemma \ref{pi inequality Olliv}, and for brevity
let $\alpha = 2n-1$. Then $\frac{p_{i}}{p_{i-1}}\leq \alpha^{-\omega_{i}}\frac{\alpha}{\alpha - 1}$.
Therefore $$-\omega_{i}\geq \log_{\alpha}p_{i}-\log_{\alpha}p_{i-1}+\log_{\alpha}(\alpha - 1)-1.$$
Hence
\begin{equation*}
	\begin{split}
			\vert \partial A\vert &\geq \; k\vert A\vert - 2\sum_{f\;face\;of\;A}\omega (f) \\
&\geq \;k\vert A\vert -2\sum_{i=1}^{N}m_{i}\omega_{i}\\
&\geq   \;k\vert   A\vert +2\sum_{i=1}^{N}m_{i}(\log_{\alpha}p_{i}-\log_{\alpha}p_{i-1}
 +   \log_{\alpha}(\alpha  - 1)-1).
\end{split}
\end{equation*}
Note that $p_0 = 1$, so
\begin{equation*}
\sum_{i=1}^{N}m_{i}(\log_{\alpha}p_{i}-\log_{\alpha}p_{i-1})= \sum_{i=1}^{N-1}(m_{i}-m_{i+1})\log_{\alpha}p_{i}+m_{N}\log_{\alpha}p_{N}.
\end{equation*}
Therefore, as $\sum\limits_{i=1}^{N}m_{i}=\vert A\vert$,
\begin{equation*}
\begin{split}
\vert \partial  A\vert & \geq    k\vert  A\vert
                          +2\sum_{i=1}^{N-1}(m_{i}-m_{i+1})\log_{\alpha}p_{i}+2m_{N}\log_{\alpha}p_{N}-2\vert
                          A\vert(1 - \log_{\alpha}(\alpha - 1)) \\
& = |A|(k - 2(1-\log_{\alpha}(\alpha - 1))) +
  2\sum_{i=1}^{N-1}(m_{i}-m_{i+1})\log_{\alpha}p_{i}+2m_{N}\log_{\alpha}p_{N}. 
\end{split}
\end{equation*}

Let $P_{i}$ be the probability that there exist $i$ relators in
$R$ partially fulfilling $A$. Notice that 
$P_{i}\leq \vert R\vert^{i} p_{i}$, so that $P_{i}\leq \alpha^{idk}
p_{i}$,
 and hence $\log_{\alpha}p_{i}\geq
 \log_{\alpha}P_{i}-idk$. Substituting, we get
\begin{equation*}
		\vert \partial A\vert \geq   |A| (k - 2(1- 
        \log_{\alpha}(\alpha - 1)))
        +2\sum_{i=1}^{N-1}(m_{i}-m_{i+1})(\log_{\alpha}P_{i}-idk)+2m_{N}(\log_{\alpha}P_{N}-Ndk). \end{equation*}
Now, \begin{equation*}
\sum_{i=1}^{N-1}(m_{i}-m_{i+1})i+m_{N}N=\sum_{i=1}^{N}m_{i} =\vert A\vert.
\end{equation*}
so we can rearrange the previous expression to get
\begin{equation*}
	\vert \partial A\vert \geq |A| (k - 2(1 - \log_{\alpha}(\alpha
        - 1)))  - 2dk|A|
                +2\sum_{i=1}^{N-1}(m_{i}-m_{i+1})\log_{\alpha}P_{i}+2m_{N}\log_{\alpha}P_{N},
\end{equation*}
Now let $P=\min\limits_{i}P_{i}$. Then from $m_{i} - m_{i+1} \geq 0$ and
$\log_{\alpha}P_i \geq \log_{\alpha}P$ we deduce that 
\begin{equation*}
\begin{split}
	\vert \partial A\vert\geq &|A| (k - 2dk - 2(1 - \log_{\alpha}(\alpha - 1)))
        +2\sum_{i=1}^{N-1}(m_{i}-m_{i+1})\log_{\alpha}P+2m_{N}\log_{\alpha}P\\
	  \geq &  |A|(k - 2dk - 2(1- \log_{\alpha}(\alpha - 1))) +2m_{1}
          \log_{\alpha}P \\
	\geq & |A|k(1 - 2d - (2/k)(1 - \log_{\alpha}(\alpha - 1)))+2\vert A\vert \log_{\alpha}P.
\end{split}
\end{equation*}
So
\begin{equation*}
 P\leq \alpha^{\frac{1}{2}\left(\frac{\vert \partial A\vert }{\vert
     A\vert}-k(1-2d- \frac{2}{k}(1 - \log_{\alpha}(\alpha - 1)))\right)}.
\end{equation*}

Now it is immediate that $P(A$ fulfillable$)\leq P$. If
\begin{equation}\label{eq:ineq}
	\frac{1}{2}\left(\frac{\vert \partial A\vert }{\vert
            A\vert}-k(1-2d- \frac{2}{k}(1 - \log_{\alpha}(\alpha - 1)))\right)\leq -\varepsilon k,
\end{equation}
then
\begin{equation*}
	P(A\;\mbox{fulfillable})\leq P\leq \alpha^{-\varepsilon k} =
        (2n-1)^{-\varepsilon k}. 
\end{equation*}
If Equation(\ref{eq:ineq}) does not hold, then
\begin{equation*}
\vert \partial A\vert > k\vert A\vert(1-2d-2\varepsilon- 
\frac{2}{k}(1 - \log_{(2n-1)}(2n-2)))
\end{equation*}
and the result follows. 
\end{proof}

Setting $\delta = \frac{\varepsilon}{2}$, letting $n$ be large enough
so that $\frac{2}{k}(1 - \log_{(2n-1)}(2n-2)))\leq \varepsilon$, and
substituting $\delta$ for $\varepsilon$ into Lemma~\ref{linisolliv}
yields the following slightly simpler statement. 

\begin{Lemma}\label{ollivlii}
		Let $k, n\geq 2$, $d<\frac{1}{2},$ and $\varepsilon
                >0$. Let $G=\langle X\;\vert\;R\rangle$ be a random
                group in the standard $(n,k,d)$ model. For large
                enough $n$ any abstract reduced van Kampen diagram,
                $A$, for $G$ either satisfies $\vert \partial A\vert
                \geq k(1-2d-2\varepsilon)\vert A\vert$ or the
                probability it is fulfillable in $G$ is less than
                $(2n-1)^{-\frac{\varepsilon}{2} k}$. 
\end{Lemma}

Lemma~\ref{ollivlii} gives us a bound on the probability that a single
abstract spurless van Kampen diagram fails to satisfy a given isoperimetric
inequality. To generalise the result to all such diagrams, we first
count them. 
The following proof is very slightly corrected from \cite[p614]{Ollivier2004}.
\begin{Lemma}\label{vkdbound}
Let $m \in\mathbb{N}$ and $G$ be a group with relators all of length
$k$. The number $N(m,k)$
of abstract spurless reduced van Kampen diagrams for $G$ with at most 
$m$ faces is bounded above by  $G(m)(2mk^4)^{m}$, 
where $G(m)$ is a constant depending only on $m$.
\end{Lemma}
\begin{proof}
Let $D$ be an abstract spurless reduced van Kampen diagram for $G$,
with boundary length $\lambda \geq 0$ and $t$ $2$-cells. 

If $t = 1$, then $\lambda \geq 1$. Such a diagram has  
one choice for relator number, $2$ choices for
orientation, and a single choice (up to equivalence)
 for the start point of the relator.
 
If $\lambda = 0$ and $t = 2$, then there are two choices for relator
numbers, four choices of orientation of the two $2$-cells, and up to
$k$ choices for the distance between the two labelled start points. 

If $t > 2$, or if $t = 2$ and $\lambda \geq 1$, then $D$ 
can be thought of as a connected planar graph with $t-1$ or $t$ faces
(including the external face, if any), with vertices
all of degree at least three, along with some extra information. By
Euler's formula, any such graph with at least $2$ faces and at most
$t$ faces 
has at most $3t$ edges. Each edge can have label length 
between $1$ and $k$. Each non-external face can have one of two 
orientations, one of $k$ start points for the relator, and one 
of at most $t$ choices of relator. So there are at most $(2kt)^{t}$ 
choices for the decoration of  all of the faces. 

For $m > 1$, let 
 $G(m)-1$ be the number of connected planar graphs with at 
most $m$ faces and with all vertices of degree at least $3$.
Then there are at most 
$(G(m)-1)(2mk^4)^{m}$ spurless reduced abstract van Kampen 
diagrams with at least two regions (or at least three regions when
spherical)  and at most $m$ regions. 
Therefore $N(m,k)\leq (G(m)-1)(2mk^4)^{m}+2 + 8k\leq G(m)(2mk^4)$.
\end{proof}

The following immediately implies the hyperbolicity claims in
Theorems~\ref{mainthm2}(i) and \ref{mainthm1}(i). 
\begin{theorem}\label{standard-hyp}
 Let $k, n\geq 2$, $d<\frac{1}{2}$, and let $G=\langle
 X\;\vert\;R\rangle$ be a  random group in an $(n,k,d)$ 
 model. 
Let $\varepsilon=(1-2d)/8$, and $C=(1-2d-2\varepsilon)$. Asymptotically almost surely 
any van Kampen diagram, $A$, for $G$  satisfies 
$|\partial A|  \geq k(C-\varepsilon)\vert A\vert,$ and so $G$ is hyperbolic.
\end{theorem}
\begin{proof}
We consider the standard model first. 
By Lemma \ref{ollivlii}, any spurless abstract reduced van Kampen for $G$ either satisfies
$\vert \partial A\vert \geq kC\vert A\vert$ or the probability it is
fulfillable is less than $(2n-1)^{-\frac{\varepsilon}{2} k}$. 

The probability that any spurless reduced van Kampen diagram for $G$
with at most $m$ faces does not satisfy the above inequality is less
than $G(m)(2mk^4)^{m}(2n-1)^{-\frac{\varepsilon}{2} k}\rightarrow
0$ as $n\rightarrow \infty$. So a.a.s. any spurless reduced van Kampen
diagram for $G$ with at most $m$ faces satisfies $\vert \partial 
A\vert \geq kC\vert A\vert$,
and so the same inequality holds for any reduced van Kampen diagram $A$ 
for $G$ with at most $m$ faces.
Hence by \cite[Theorem 8]{ollivier2007}, a.a.s. any reduced van 
Kampen diagram in $G$ satisfies $\vert \partial A\vert \geq
(C-\varepsilon)k\vert D\vert$,
and so
$G$ is hyperbolic.

For the positive model, we argue as in Lemma \ref{linisolliv},
        but replace Lemma~\ref{pi inequality Olliv} with
 $p_{i}\leq p_{i-1}n^{-\omega_{i}},$ to deduce that
        any abstract reduced van Kampen diagram $A$ for $G$ either satisfies 
$
\vert \partial A\vert \geq k\vert A\vert(1-2d-2\varepsilon)
$
	or the probability it is fulfillable in $G$ is less than
        $n^{-\varepsilon k}$. 
	The result then follows exactly as for the standard model. 
\end{proof}

Finally, we show that at density less than $1/2$ our groups are
infinite and torsion-free, and looking ahead to the next section we consider
freeness. 

\begin{proposition}\label{prop:free-infinite}
Let $k \geq 2$, and let $G$ be a random group in an $(n, k, d)$ model.
If $0 < d < \frac{1}{2}$ then asymptotically almost surely $G$ is
infinite and
torsion-free. 
If $\frac{1}{k} <  d < \frac{1}{2}$ then asymptotically almost surely
$G$ is not isomorphic to a free group. 
\end{proposition}

\begin{proof}
We showed in Theorem~\ref{standard-hyp} that
at density $d < 1/2$ the group $G$ satisfies a linear
isoperimetric inequality with additive constant zero, which implies
that there are no van Kampen diagrams with boundary length $0$. Our definitions
of van Kampen diagrams were sufficiently general that this implies that
the Cayley 2-complex is aspherical. This implies that $G$ has 
cohomological dimension at most $2$, and so is torsion free. 

The 
only finite group that is torsion-free is the trivial group, however 
the Euler characteristic of $G$ is equal to $1 - n + (2n-1)^{dk}$ in
the standard model and $1 - n + n^{dk}$ in the positive model. 
For $d > 1/k$ this is greater than 1, whereas the trivial group has 
Euler characteristic $1$ and a free group of rank $\ell$ has 
Euler characteristic $1 - \ell$. 
\end{proof}

This concludes the proof of Theorems~\ref{mainthm2}(ii) and
\ref{mainthm1}(ii). 


\section{Groups at density less than $1/k$}\label{sec: freeness}

In this section we prove that a group $G$ in either $k$-angular
model is asympotically almost surely free if the density is less than
$1/k$. The fact that this bound is tight follows from
Proposition~\ref{prop:free-infinite}. 

We start by proving a sufficient condition on the presentation complex
for a group to be free.

\begin{definition}
	Let $G=\langle X\;\vert\; R\rangle$ be a group
        presentation. We construct the \emph{presentation complex},
        $\mathcal{P}$, for $G$ as follows. Take a single vertex, $v$,
        as the only $0$-cell. Take as the $1$-cells oriented loops $x_{i}$ at
        $v$ for each $x_{i}\in X$. Take a $2$-cell for each relator,
        $r_{i}$, with the boundary mapped to the succession of
        appropriately oriented
        $1$-cells $x_{i_{1}}^{\epsilon_1},\hdots,
        x_{i_{n}}^{\epsilon_n}$ with $\epsilon \in \{\pm 1\}$ such that
        $r_{i}=x_{i_{1}}^{\epsilon_1}\hdots x_{i_{n}}^{\epsilon_n}$.\end{definition} 

It is standard that $\Pi_1(\mathcal{P}) \cong G$. We remark that in an
$(n, k, d)$ 
model we can define the
closed $2$-cells of $\mathcal{P}$ to be isometric with the closed unit
$k$-gon in Euclidean space, and assign
the path metric to $\mathcal{P}$ to turn $\mathcal{P}$ into a metric
complex. 

The next set of definitions are based on those of \cite{ollivier2011cubulating}.
\begin{definition}
	Let $\mathcal{C}$ be a path-connected $2$-complex such that
        all $2$-cells have even boundary length. We construct a labelled undirected graph,
        $\Gamma_{\mathcal{C}}$, called the
        \emph{antipodal graph} for $\mathcal{C}$. Take as the vertices
        the set of $1$-cells in $\mathcal{C}$. For each $2$-cell $F$,
        and for each pair $v_1, v_2$ of (not necessarily distinct)
        $1$-cells that are antipodal on $\partial(F)$, add an
        edge $\{v_1, v_2\}$ to $\Gamma_{\mathcal{C}}$ labelled $F$. We
        say that $F$ is the $2$-cell 
        \emph{containing} the edge
        $\{v_1, v_2\}$.

If $\mathcal{C}$ is a metric complex, then 
there is a map, $\phi : \Gamma_{\mathcal{C}} \rightarrow 
\mathcal{C}$, as follows. The map $\phi$ sends each vertex of
$\Gamma_{\mathcal{C}}$  to 
the midpoint of the corresponding $1$-cell of $\mathcal{C}$, and sends
each edge $\{v_1, v_2\}$ of $\Gamma_{\mathcal{C}}$,  labelled $F$, to a 
(non self-intersecting) path in $F\setminus \partial F$ joining $\phi(v_1)$ and $\phi(v_2)$. 
\end{definition}

\noindent 
\begin{definition}
	Let $\mathcal{C}$ be a path-connected $2$-complex such that
        all $2$-cells have odd boundary length. Form the \emph{halved
          complex} $H_{\mathcal{C}}$ of $\mathcal{C}$ as
        follows. Replace each $1$-cell, $e$, in $\mathcal{C}$, by two 
         $1$-cells $e_{1}$ and $e_{2}$, meeting at a $0$-cell. The
         edge 
         $e$ is the \emph{precursor} of $e_{1}$ and $e_{2}$.  
       The antipodal graph for $H_{\mathcal{C}}$ is the \emph{halved
         antipodal graph} of $\mathcal{C}$; we write $\Gamma^{H}_{\mathcal{C}}:=\Gamma_{H_{\mathcal{C}}}$.

 If $\mathcal{C}$ is metric, then there is a map
 $\phi^{H}:\Gamma^{H}_{\mathcal{C}}\rightarrow \mathcal{C}$, as
 follows. The map $\phi$ sends each vertex $v$ to the midpoint of
 the precursor of $v$ in $\mathcal{C}$, and sends each edge $\{v_1,
 v_2\}$ of $\Gamma^H_{\mathcal{C}}$, labelled $F$, to a (non
 self-intersecting) path in
 $F \setminus \partial F$ between $\phi(v_1)$ and
 $\phi(v_2)$. 
\end{definition}

Notice that the graph $\Gamma^{H}_{\mathcal{C}}$ consists of two disjoint components.

\begin{definition}
Let $\mathcal{C}$ be a path-connected 2-complex such that the boundary
lengths of the 2-cells are either all odd, or all even.  
Let $\Gamma$ be $\Gamma_{\mathcal{C}}$ in the even case, or
 $\Gamma^H_{\mathcal{C}}$ in the odd case, and let $\varphi = \phi$ in
 the even case, or $\varphi = \phi^H$ in the odd case.  
	A \emph{hypergraph} $\Lambda$ of $\Gamma$ is a connected component of
        $\Gamma$.
We also refer to $\varphi(\Lambda)$ as a hypergraph of $\mathcal{C}$, and
if $\Lambda$ is a tree, then we will also refer
to $\varphi(\Lambda)$ as a tree. 
\end{definition}

 \begin{theorem}\label{hyptreefree}
 Let $G=\langle X\;\vert\;R\rangle$ be a finite presentation with all
 relators 
of length $k$, such that all hypergraphs in the presentation complex
$\mathcal{P}$ of $G$
are trees. Then $G$ is a free group.
 \end{theorem}

\begin{proof}
Let $\Gamma = \Gamma_{\mathcal{P}}$ when $k$ is even, and
$\Gamma = \Gamma^H_{\mathcal{P}}$ when $k$ is odd. We show by
induction on the number of edges of $\Gamma$ that $G$ is free.

The result is clear when $\Gamma$ has no edges, as then $\mathcal{P}$
is just a bouquet of circles. Assume therefore that $\Gamma$ contains
at least $i>1$ edges. 
 Let $\Lambda$ be a connected component in $\Gamma$ containing at least one edge.
Then $\Lambda$ is a tree, and so contains at least one vertex of valency $1$,
which must correspond to a $1$-cell in $\mathcal{P}$ which is on the
boundary of a single $2$-cell. This means that there is a letter 
$a \in X$ which occurs in a single relator $r \in R$, and we can apply a Tietze
transformation to the presentation which replaces $X$ by $X' = X
\setminus \{a\}$ and replaces $R$ by $R' = R \setminus r$ (all other
relators are unchanged). Then $G \cong \langle X' \mid R'
\rangle$. 

Let $\mathcal{P}'$ be the presentation complex of this new
presentation. If $k$ is even
then $\Gamma' = \Gamma_{\mathcal{P}'}$ has vertex set $X \setminus a$, and edge
set a subset of $E(\Gamma)$ with precisely $k > 1$ edges removed. If
$k$ is odd then $\Gamma' = \Gamma_{\mathcal{P}'}$ has vertex set a
subset of $X$ with the two precusors of $a$ removed, and edge set a
subset of $E(\Gamma)$ with precisely $2k > 1$ edges removed. 
Thus all connected components of $\Gamma'$ are trees, and each such
component embeds in $\mathcal{P}$. 
Hence the result follows by induction.
\end{proof}

We now introduce two further variants on van Kampen diagrams, which
will give us a necessary condition for all hypergraphs in a
presentation complex to be trees.

\begin{definition}
Let $G=\langle X\;\vert\;R\rangle$ be a presentation. A \emph{relator
  diagram} for $G$ is defined in exactly the same way as a van Kampen
diagram, except that we permit the complex to be annular, and we
permit it to be homeomorphic to a M\"obius strip. 

 An \emph{abstract relator diagram} is the $2$-complex obtained from a 
relator diagram $A$ in the same way as an abstract van Kampen diagram
is obtained from a van Kampen diagram. That is,  
for each $2$-cell $f$ bearing relator $r_{i_{j}}$, label the $2$-cell with
the number $j$, remember only the starting point of the relator and the orientation of 
the relator on $\partial f$.
 The definition of relators \emph{fulfilling} an abstract relator
 diagram  follows as for an abstract reduced van Kampen diagram.
\end{definition} 

\begin{definition}
	Let $G=\langle X\;\vert\;R\rangle$ be a finite presentation with all relators of length $k$, and let $\ell\geq 2$. Let $D$ be a relator diagram for $G$ with the following properties.
		\begin{enumerate}[label=\roman*)]
			\item $D$ has $\ell$ $2$-cells, $F_{1},\hdots,
                          F_{\ell}$, each bearing a distinct relator,  and $(k-1)\ell$ $1$-cells.
			\item  For $2\leq j\leq \ell-1$, $\partial F_{j}$ shares a single $1$-cell with $\partial F_{j-1}$, another with $\partial F_{j+1}$, and none with any other $2$-cells.
			\item $\partial F_{1}$ shares a single $1$-cell with $\partial F_{\ell}$.
		\end{enumerate}
Then $D$ \emph{satisfies Condition $\dagger (\ell)$}.
\end{definition}

\begin{Lemma}\label{daggeriimplesnononembedhyp}
	Let $G=\langle X\;\vert\;R\rangle$ be a finite presentation
        with all relators of length $k$, and let $\mathcal{P}$ be the
        presentation complex of $G$. Suppose $G$ has no relator diagrams satisfying Condition $\dagger (\ell)$ for any $\ell\geq 2$, and is such that no letter in $X^{\pm}$ (or its inverse) appears more than once in any relator. Then all hypergraphs in $\mathcal{P}$ are trees.
\end{Lemma}

\begin{proof}
We prove this by contradiction. If $k$ is even, then let $\Gamma =
\Gamma_{\mathcal{P}}$ and $\varphi = \phi$. If $k$ is odd, then let
$\Gamma = \Gamma_{\mathcal{P}}^H$ and $\varphi = \phi^H$. 

 Suppose there exists a hypergraph $\Lambda$ in $\mathcal{P}$
 such that $\phi(\Lambda)$ contains a circuit. This means that there is a
 path $e_{1}, \hdots,e_{m}$ 
of edges in $\Gamma$, with $m\geq 1$, such that $\phi$ is
not injective when restricted to this path.  
 We can assume that
 this path is of minimal length, so that $e_{i}\neq e_{j}$ for $i\neq
 j$: this may mean that the path is not a cycle in $\Gamma$. 

If $m=1$, then
 the edge $e_1$ is a loop in $\Gamma$, since
 $\varphi(e_1)$ can only self-intersect at its end-points.
However this means that a letter appears more than once in a relator,
a contradiction.
	
	For $m\geq 2$, form the following relator diagram $D$. 
 Pick as the first $2$-cell $F_{1}$ the (unique) $2$-cell which contains
 $e_{1}$. Now for $2 \leq i \leq m$, pick $F_{i}$ as the 
 $2$-cell containing $\varphi(e_{i})$. The minimality of $m$ ensures
 that each $F_i$ is distinct, except possibly $F_m = F_1$ when $m
   \geq 3$,
 and it is clear that $D$ satisfies the three
 requirements of Condition $\dagger(m)$ if $F_m \neq F_1$ and
 $\dagger(m-1)$ otherwise. 
By assumption such a relator diagram does not exist, and so such a $\Lambda $ does not exist.
	\end{proof}

We now consider the first necessary condition for freeness in Lemma~\ref{daggeriimplesnononembedhyp}.
\begin{Lemma}\label{probdaggeriexistnormal}

	Let $k \geq 2$, let $0< d<\frac{1}{k}$, and let $G=\langle
        X\;\vert\;R\rangle$ be a  random group in an $(n,k,d)$
        model. Let $\ell \geq 2$, and let $P_{\ell}$ denote the probability that an
        abstract relator diagram which gives rise to a relator diagram
        satisfying Condition $\dagger (\ell)$ is fulfillable in $G$.
 
There exists a constant $c:= c(d, k)<0$, depending on
        $d$, $k$ and the choice of model, such that for sufficiently
        large $n$, the value of $P_{\ell}$ is
        at most $(2n-1)^{c\ell}$ for the standard model, and at
        most $n^{c\ell}$ for the positive model. 
\end{Lemma}
\begin{proof}
First consider the standard model. 
Let $D_{\ell}$ be an abstract relator diagram
which gives rise to a relator diagram satisfying Condition $\dagger
(\ell)$. 

Let $R_{\ell}$ denote the set of distinct $\ell$-tuples of cyclically reduced
words fulfilling $D_{\ell}$: we first determine $|R_{\ell}|$. 
There are $(k-1)\ell$ $1$-cells in $D_{\ell}$, the first of which may
be labelled freely whilst the rest have at most $2n-1$ possible
labels. Hence
$|R_{\ell}| \leq (2n)(2n-1)^{(k-1)\ell-1}$, and so for any $\delta_1 >
0$, for $n$ sufficiently large we can bound
$$|R_{\ell}| \leq (2n-1)^{(k-1+\delta_{1})\ell}.$$ 

For $\alpha\in R_{\ell}$, let $P_{\alpha}$ be the probability that
$\alpha\subseteq R$; that is, the probability that these particular
$\ell$ cyclically reduced words are in $R$. Then
	\begin{equation*}
	P_{\ell}\leq \sum\limits_{\alpha\in R_{\ell}}P_{\alpha}.
	\end{equation*}
	Recall that $C_{k,n}$ denotes the set of cyclically reduced
        words of length $k$ in $F_{n}$. We see that
	\begin{equation*}
P_{\alpha}=\dfrac{{\dbinom{\vert C_{k,n}\vert-\ell}{(2n-1)^{kd}-\ell}}}{{\dbinom{\vert C_{k,n}\vert}{(2n-1)^{kd}}}}
	=\dfrac{(2n-1)^{kd}((2n-1)^{kd}-1)\hdots ((2n-1)^{kd}-\ell+1)}{\vert C_{k,n}\vert(\vert C_{k,n}\vert-1)\hdots(\vert C_{k,n}\vert-\ell+1)}
	 <\frac{(2n-1)^{\ell kd}}{\vert C_{k,n}\vert^{\ell}},
	\end{equation*}
since $|C_{k, n}| > (2n-1)^{kd}$. 
Now, for any $\delta_{2}>0$, for $n$ sufficiently large
$\vert C_{k,n}\vert >(2n-1)^{k-\delta_{2}}$.  Therefore,
for $n$ sufficiently large  we can bound $P_{\alpha} <  (2n-1)^{(kd-k+\delta_{2})\ell}$, and so 
\begin{equation*}
	P_{\ell}\leq \vert R_{\ell}\vert P_{\alpha}< (2n-1)^{(k-1+\delta_{1})\ell}(2n-1)^{(kd-k+\delta_{2})\ell}=(2n-1)^{(kd-1+\delta_{1}+\delta_{2})\ell},
\end{equation*}
for all $\delta_{1},\delta_{2}>0$. 

Since $kd-1<0$, we can 
choose $\delta_{1},\delta_{2}>0$ such that $c= c(d, k) 
= kd-1+\delta_{1}+\delta_{2}<0$. We then conclude that for $n$
sufficiently large we can bound
$P_{\ell}\leq (2n-1)^{c\ell}$,
as required. 

The proof for the positive model is similar but easier:
$|R_{\ell}| = n^{(k-1)\ell}$  and $P_{\alpha}  < n^{k \ell (d-1)}$, so
  $P_{\ell} \leq n^{(kd - 1)\ell}$, and we can set $c = c(d, k) = kd -
  1 < 0$. 
\end{proof}

\begin{theorem}\label{normhyptrees}
	Let $k\geq 2$, $0 < d<\frac{1}{k}$, and $G=\langle
        X\;\vert\;R\rangle$ be a
 random group in an $(n,k,d)$ model. Asymptotically almost surely $G$ has no relator diagrams satisfying Condition $\dagger (\ell)$ for any $\ell\geq 2$.\end{theorem}

\begin{proof}
First we show that the number of abstract relator diagrams giving rise to a relator diagram for $G$ satisfying Condition $\dagger (\ell)$ is at most $(2k)^{\ell}$.
Due
          to Condition $\dagger (\ell)$, every face of such an
          abstract relator diagram is labelled by a different number.
        Hence there is (up to equivalence) only one way to choose the
        labels. There are $2$ choices of orientation for each face,
        and $k$ choices of start point for each relator. Hence there
        are
 at most $(2k)^{\ell}$ such abstract relator diagrams.

Now let $m = 2n-1$ for the standard model, and $m = n$ for the positive model.
By Lemma \ref{probdaggeriexistnormal}, there exists a constant $c< 0$
such that for $m$ sufficiently large the probability
that a relator diagram satisfying Condition $\dagger (\ell)$ exists
for $G$ is at most $(2k)^\ell m^{c\ell}$. Since $2km^{c}$ tends to $0$ as
$m \rightarrow \infty$, the probability that such a diagram exists in
$G$ for any $\ell$ is at most
$$
	\sum\limits_{\ell=2}^{\infty}(2km^{c})^{\ell}\leq  \sum\limits_{\ell=0}^{\infty}(2km^{c})^{\ell}-1 = \frac{1}{1-2km^{c}}-1
$$
which tends to $0$ as $m$ tends to $\infty$.
\end{proof}

\begin{proof}[Proof of Theorem~\ref{mainthm3}]
First we show that asymptotically almost surely there are no repeated
letters from $X$ in any relator of $R$. 
	Consider the standard $(n,k,d)$ model. The probability that a
        fixed letter (or its inverse) appears twice in a fixed relator
        is at most
        $\frac{k(k-1)}{2} \frac{4}{(2n-2)^{2}}$, since at each position in the
        relator we have at least $2n-2$ choices for the letter.  The
        size of $|X| = n$, and there are 
        $(2n-1)^{dk}$ relators,  so the probability that any letter
        appears more than  once in a relator is at most
	$$\frac{k(k-1)2n(2n-1)^{dk}}{(2n-2)^{2}}\rightarrow 0\mbox{ as }n\rightarrow\infty.
	$$ The proof
        for the positive model is similar but easier. 

By Theorem~\ref{normhyptrees}, asymptotically almost surely $G$ has no
relator diagrams satisfying Condition $\dagger(\ell)$ for any $\ell
\geq 2$, so by Lemma~\ref{daggeriimplesnononembedhyp}, asymptotically
almost surely all hypergraphs in the presentation complex
$\mathcal{P}$ of $G$ are embedded trees. 
The result now follows from Theorem~\ref{hyptreefree}.
\end{proof}

\section{Appendix: proof of Lemma~\ref{red_connected}}

We thank Louis Theran for sketching out some of the ideas in this
section to us. 

\begin{definition}
An \emph{Erd\"{o}s-R\'enyi  random bipartite graph} $\Gamma(a,b,p)$ is
a bipartite graph with $\vert V_{1}(\Gamma)\vert=a$, $\vert
V_{2}(\Gamma)\vert =b$, and with each edge added with probability
$p$. \end{definition}

First we establish a lower bound of the degree of each vertex in
$V_1$. 

\begin{Lemma}\label{erdosrenyiconnected}
	Let $d>\frac{1}{2}$, let $m\geq 1$, fix
        $0<\varepsilon<d-\frac{1}{2}$, and let $p=n^{(2m+1)(d-1)}$. With
        probability tending to $1$ as $n\rightarrow\infty$, all
        vertices in $V_{1}$ in $\Gamma
        (n^{m},n^{m+1},p)$ have degree at least
        $k=n^{\frac{1}{2}+\varepsilon (2m+1)}$. 
\end{Lemma}
\begin{proof}
Fix $v \in V_1$. The degree $d(u)$  of a vertex 
$u\in V_{1}$ satisfies $ d(u)\sim B(n^{m+1},p)$. Furthermore the
degrees of the vertices are independent, so the probability $P_{k}$
that there exists a
vertex in $V_{1}$ of degree less than $k$
is $P_k = n^{m}P(d(v) < k)$. Let $\lambda = d - 1/2$, so that $\lambda
> \varepsilon$. 
Notice that $k = n^{1/2 + \varepsilon(2m+1)}$, whilst the expected
value of $d(v)$ is $n^{m+1}p = n^{1/2 +
  \lambda(2m+1)}$, so  
we can use the multiplicative lower tail form of Chernoff's
 inequality to bound $$P(d(v) < k)=P(d(v)\leq k-1) \leq \exp\{
 -\frac{(n^{m+1}p-(k-1))^{2}}{2n^{m+1}p}\},$$ and so  
$P_{k}\leq n^{m} \exp\{-\frac{(n^{m+1}p-(k-1))^{2}}{2n^{m+1}p} \},$
which tends to $0$ as $n \rightarrow \infty$. 
\end{proof}

The following standard result is proved by by approximating
the binomial distribution $Bin(t, p)$ by the normal distribution
$N(tp, tp(1-p))$.

\begin{Lemma}\label{binnote}
	Let $X\sim Bin(t,p)$, such that $p$ is a function of $t$, and $\sqrt{tp}=o(tp)$. Then with probability tending to $1$ as $t$ tends to $\infty$, $X\in[\frac{1}{2}tp,\frac{3}{2}tp]$.
\end{Lemma}

\begin{Lemma}\label{bipartitegraphconn}
Let $d>\frac{1}{2}$, $m\geq 2$,
                $0<\varepsilon<d-\frac{1}{2}$,
                $k=n^{\frac{1}{2}+\varepsilon (2m+1)}$, and
                $E=n^{(2m+1)d}$. Let $\mathcal{K}$ be the set of
                bipartite random graphs such that all vertices in
                $V_{1}$ have degree at least $k$. Then with
                probability tending to $1$ as $n$ tends to infinity,
                $\Gamma (n^{m},n^{m+1},E)$ is in $\mathcal{K}$. 
\end{Lemma}
\begin{proof}	
Let $p = \frac{2}{3}n^{(2m+1)(d-1)}$, so that that $\frac{3}{2}n^{2m+1}p=E$. 
Then there exists a $d'\in (0,d)$ such that $p\geq
n^{(2m+1)(d'-1)}$, so 
by Lemma \ref{erdosrenyiconnected}, 
$P(\Gamma(n^{m},n^{m+1},p)\in\mathcal{K})\rightarrow1$ as 
$n\rightarrow \infty$. 

Write $E_p$ for $\vert
E(\Gamma(n^{m},n^{m+1},p))\vert$. Since
$\sqrt{n^{2m+1}p}=o(n^{2m+1}p)$, by
Lemma \ref{binnote}, 
\begin{equation*}
	\delta_{n}:=P\left(E_p \not \in [\frac{1}{2}n^{2m+1}p, \frac{3}{2}n^{2m+1}p]\right)
\end{equation*}
satisfies $\delta_{n}=o(1)$. Hence, since the probability of being in $\mathcal{K}$ is only increased by adding edges,
\begin{equation*}
\begin{split}
P(\Gamma(n^{m},n^{m+1},p)\in\mathcal{K})=&\sum\limits_{i=0}^{n^{2m+1}}P(\Gamma(n^{m},n^{m+1},i)\in\mathcal{K})P(E_p=i)\\
		\leq & \sum\limits_{i=\frac{1}{2}n^{2m+1}p}^{\frac{3}{2}n^{2m+1}p}P(\Gamma(n^{m},n^{m+1},i)\in\mathcal{K})P(E_p=i)+\delta_{n}
		\\\leq & \sum\limits_{i=\frac{1}{2}n^{2m+1}p}^{\frac{3}{2}n^{2m+1}p}P(\Gamma(n^{m},n^{m+1},E)\in\mathcal{K})P(E_p=i) +\delta_{n}\\
		\leq & P(\Gamma(n^{m},n^{m+1},E)\in\mathcal{K})+\delta_{n}.		
\end{split}	
\end{equation*}	
Therefore, $P(\Gamma(n^{m},n^{m+1},E)\in\mathcal{K})\geq P(\Gamma(n^{m},n^{m+1},p)\in\mathcal{K})-\delta_{n}\rightarrow 1\mbox{ as }n\rightarrow \infty$.
\end{proof}

\begin{proof}[Proof of Lemma~\ref{red_connected}]
 Let
$0<\varepsilon<d-\frac{1}{2}$. Then by Lemma \ref{bipartitegraphconn},
 with probability tending to $1$ as
as $n \rightarrow \infty$,  each vertex in $V_1$ has degree at least
$\frac{1}{2} n^{1/2 + \varepsilon(2m+1)}$.

Now, let $s_1, s_2 \in V_1$, and consider the neighbourhoods $N(s_1), N(s_2) \subseteq V_2$. We calculate
$$\begin{array}{rl}
P(N(s_1) \cap N(s_2)  = \emptyset) & \leq \left( 1 - \frac{\frac{1}{2} n^{1/2 + \varepsilon(2m+1)}}{n^{m+1}} \right)^{\frac{1}{2} n^{1/2 + \varepsilon(2m+1)}} \\
& = 
\left( 1 - \frac{1}{2}n^{-m - 1/2 + \varepsilon(2m+1)}\right)^{\frac{1}{2} n^{1/2 + \varepsilon(2m+1)}}\\
& \leq \left( e^{-\frac{1}{2}n^{-m  -1/2 +
  \varepsilon(2m+1)}}\right)^{\frac{1}{2} n^{1/2 + \varepsilon(2m+1)}}\\
& = e^{-\frac{1}{4}n^{-m + 2\varepsilon(2m+1)}} \\
& \leq 1 - \frac{1}{4}n^{-m + 2 \varepsilon(2m+1)} + \frac{1}{32}n^{-2m + 4 \varepsilon(2m+1)}.
\end{array}$$
Hence the probability that the two neighbourhoods intersect is at least
$$\frac{1}{4}n^{-m + 2\varepsilon(2m+1)} - \frac{1}{32}n^{-2m + 4 \varepsilon(2m+1)} = \frac{8n^{m - 2\varepsilon(m+1)} - 1}{32n^{2m - 4 \varepsilon(2m+1)}}.$$
Choose $\gamma$ such that $0 < \gamma < 2 \varepsilon(2m+1)$, and let
$\mu = 2\varepsilon(2m+1) - \gamma > 0$. Then for sufficiently large $n$ the probability that the two neighbourhoods intersect is at least
$$\frac{8 n^{m- 2\varepsilon(2m+1) -
    \gamma}}{32n^{2m-4\varepsilon(2m+1)}} =
\frac{1}{4n^{m-\mu}}. $$

Now define a graph $\Gamma'$ with vertices $V_1$, and an edge between two vertices in $\Gamma'$ if and only if the corresponding vertices in $\Gamma$ have a shared neighbour. Then $\Gamma'$ is a random graph on $n^m$ vertices, with edge density at least  $\frac{1}{4n^{m - \mu}}$, and so by standard results due to  Erd\"os and R\'enyi \cite{ErdosRenyi} the graph $\Gamma'$ is connected. The result follows.
\end{proof}

\noindent
Addresses:\\[2mm]
St Cross College\\
Oxford\\
OX1 3LZ\\
email: calum.ashcroft@stx.ox.ac.uk\\[5mm]
School of Mathematics and Statistics\\
University of St Andrews\\ St Andrews\\
Fife KY16 9SS\\
email: colva.roney-dougal@st-andrews.ac.uk

\end{document}